\documentclass[a4paper]{amsart}

\usepackage{stackrel}

\newcommand\Op{-\Delta^{\Omega^\ext}_\aa}
\newcommand\OpB{-\Delta^{B_R^\ext}_\aa}
\newcommand\Ev{\lm_1^\aa(\Omega^\ext)}
\newcommand\EvB{\lm_1^\aa(B^\ext_R)}

\usepackage{lmodern} 
\usepackage{bm,amsfonts,amsmath,amssymb,mathrsfs,bbm,amsthm} 
\usepackage{graphicx} 
\usepackage{color}
\usepackage{hyperref}

\usepackage{scrextend}
\changefontsizes{11pt}

\hypersetup{colorlinks=true, linkcolor=blue,citecolor=citegreen}
\definecolor{citegreen}{rgb}{0.2,0.2,0.6}

\newcommand{\beq}{\begin{equation} \begin{split}}
\newcommand{\eeq}{\end{split} \end{equation}}

\renewcommand\and{\qquad\text{and}\qquad}

\newcommand{\comm}[1]{}

\def\sfH{\mathsf{H}}

\def\bm1{\mathbbm{1}}

\def\s{\sigma}
\def\sfh{\mathsf{h}}

\def\p{\partial}

\def\omg{\omega}

\renewcommand{\iff}{\textit{if, and only if,}\,}

\def\arr{\rightarrow}

\def\s{\sigma}

\def\sess{\sigma_{\rm ess}}

\def\p{\partial}

\def\kp{\kappa}

\def\sfH{\mathsf{H}}
\def\sfh{\mathsf{h}}

\def\dd{{\mathsf{d}}}

\def\omg{\omega}

\def\sfU{\mathsf{U}}

\newcounter{counter_a}
\newenvironment{myenum}{\begin{list}{{\rm(\roman{counter_a})}}%
{\usecounter{counter_a}
\setlength{\itemsep}{1.ex}\setlength{\topsep}{0.8ex}
\setlength{\leftmargin}{5ex}\setlength{\labelwidth}{5ex}}}{\end{list}}

\usepackage[latin1]{inputenc}
\usepackage[T1]{fontenc}

\numberwithin{figure}{section}
\numberwithin{equation}{section}
\theoremstyle{plain}
\newtheorem*{thm*}{Theorem}
\newtheorem{thm}{Theorem}

\newtheorem{prop}[thm]{Proposition}

\newtheorem{cor}[thm]{Corollary}

\theoremstyle{remark}

\theoremstyle{plain}

\newcommand{\beu}{\begin{equation*}}
\newcommand{\eeu}{\end{equation*}}
\newcommand{\besu}{\begin{equation*}
\begin{aligned}}
\newcommand{\eesu}{\end{aligned}
\end{equation*}}
\newcommand{\bes}{\begin{equation}
\begin{aligned}}
\newcommand{\ees}{\end{aligned}
\end{equation}}

\newcommand\cH{\mathcal H}

\newcommand\cL{\mathcal L}
\newcommand\cM{\mathcal M}

\newcommand\ov{\overline}

\newcommand\void[1]{}

\def\ov{\overline}

      \def\dC{{\mathbb C}}

   \def\dN{{\mathbb N}}   
      \def\dR{{\mathbb R}}

   \def\sfH{{\mathsf H}}

      \def\sfU{{\mathsf U}}

   \def\cH{{\mathcal H}}   
      \def\cL{{\mathcal L}}
\def\cM{{\mathcal M}}

   \def\cW{{\mathcal W}}

\newcommand{\lm}{\lambda}
\newcommand{\ie}{\emph{i.e.}}
\newcommand{\eg}{\emph{e.g.}}
\newcommand{\cf}{\emph{cf.}}

\newcommand{\Real}{\mathbb{R}}

\newcommand{\Dom}{\mathsf{D}}

\newcommand{\der}{\mathrm{d}}
\renewcommand{\aa}{\alpha}
\newcommand\ext{\mathrm{ext}}
\theoremstyle{definition}

\usepackage[normalem]{ulem}
\definecolor{DarkGreen}{rgb}{0,0.5,0.1} 

\newcommand\soutD{\bgroup\markoverwith
{\textcolor{DarkGreen}{\rule[.5ex]{2pt}{1pt}}}\ULon}

\newcommand{\dint}[2][]{\displaystyle{\int_{#2}^{#1}}}
\makeatletter
\newcommand{\vast}{\bBigg@{3}}
\newcommand{\Vast}{\bBigg@{5}}
\makeatother

\begin{document}
%
\title[Optimisation of the Robin eigenvalue]{Optimisation of the lowest Robin eigenvalue in the exterior of a compact set, II: non-convex domains and higher dimensions}
\author{David Krej\v{c}i\v{r}{\'\i}k} 
\address{Department of Mathematics, Faculty of Nuclear Sciences and 
	Physical Engineering, Czech Technical University in Prague, Trojanova 13, 12000 Prague 2,
	Czech Republic}
\email{david.krejcirik@fjfi.cvut.cz}

\author{Vladimir Lotoreichik}
\address{Department of Theoretical Physics, Nuclear Physics Institute, 
	Czech Academy of Sciences, 25068 \v Re\v z, Czech Republic}
\email{lotoreichik@ujf.cas.cz}

\begin{abstract}
We consider the problem of geometric optimisation of the lowest eigenvalue
of the Laplacian in the exterior of a compact set 
in any dimension,
subject to attractive Robin boundary conditions.

As an improvement upon our previous work~\cite{KL1},
we show that under either a constraint of fixed perimeter or area, 
the maximiser within the class of exteriors of simply connected planar sets
is always the exterior of a disk, without the need of convexity assumption.
Moreover, we generalise the result to 
disconnected compact planar sets.
Namely, we prove that under a constraint of fixed average value of the perimeter over all the connected components,
the maximiser within the class
of disconnected compact planar sets, consisting of finitely many simply connected components, is again a disk.

In higher dimensions, we prove a completely new result
that the lowest point in the spectrum
is maximised by the exterior of a ball among all sets exterior to bounded convex sets satisfying a constraint on the integral of a dimensional power of the mean curvature of their boundaries. Furthermore, it follows that the critical coupling at which the lowest point in the spectrum becomes a discrete eigenvalue emerging from the essential spectrum is minimised 
under the same constraint by the critical coupling for the exterior of a ball.

\end{abstract}
\date{7 July 2017}


\maketitle

\section{Introduction}

\subsection{Motivation and state of the art}
Spectral optimisation problems constitute an intensively studied area of modern mathematics. 
In addition to important applications in physics,
the attractiveness is certainly caused by the emotional 
impacts geometric shapes have over a person's perception of the world.
Moreover, the problems are typically easy to state
but difficult to solve, leading thus to mathematically challenging interaction of differential geometry, operator theory, and partial differential equations.	
We refer to Henrot's monographs~\cite{Henrot, Henrot2} 
for many results, open problems in this area of mathematics and further references.

In this paper we are concerned 
with the optimisation of the lowest point $\lambda_1^\alpha(\Omega)$ of the spectrum
of the Robin Laplacian that is variationally characterised by the formula
\begin{equation}\label{Rayleigh}
  \lambda_1^\alpha(\Omega)
  := \inf_{\stackrel[u\not=0]{}{u \in W^{1,2}(\Omega)}}
  \frac{\displaystyle \int_\Omega |\nabla u|^2 + \alpha \int_{\partial\Omega} |u|^2}
  {\displaystyle \, \int_{\Omega} |u|^2}
  \,.
\end{equation}
We are interested in extremal properties of $\lambda_1^\alpha(\Omega)$
as regards the geometry of the smooth open set $\Omega \subset \Real^d$
and the value of the real parameter~$\alpha$. 

If~$\Omega$ is bounded,
then the infimum in~\eqref{Rayleigh} is achieved 
and $\lambda_1^\alpha(\Omega)$ represents the lowest
eigenvalue of the Robin Laplacian. 
For all positive~$\alpha$, it is then known that
$\lambda_1^\alpha(\Omega)$ is minimised 
by the ball among all domains of given volume 
\cite{Bossel_1986,Daners_2006,Daners_2013}.
For negative~$\alpha$, 
where it makes sense to look at a reverse optimisation,
it was conjectured by Bareket in 1977 \cite{Bareket_1977}
that $\lambda_1^\alpha(\Omega)$ is now maximised
by the ball among all domains of given volume 
(isochoric problem).
This conjecture has been recently disproved by Freitas 
and one of the present authors \cite{FK7}
by showing that spherical shells
give larger values of $\lambda_1^\alpha(\Omega)$
for all sufficiently large negative~$\alpha$.
In the two-dimensional situation, however,
it is true that $\lambda_1^\alpha(\Omega)$ is maximised
by the disk among all planar domains of given area
provided that~$\alpha$ is negative and small \cite{FK7}.
Moreover, numerical simulations suggest \cite{AFK}
that the conjecture actually does hold for all negative~$\alpha$
among the class of simply connected planar domains, 
but the proof constitutes a challenging open problem.
Finally, it was shown in \cite{AFK} that, for all negative~$\alpha$,
the eigenvalue $\lambda_1^\alpha(\Omega)$ 
is maximised by the disk among all planar domains of given perimeter
(isoperimetric problem).
 
The question of spectral optimisation 
is also natural to ask for unbounded sets. 
In our preceding paper~\cite{KL1},  we considered
optimisation of $\Ev$
with $\Omega^\ext:= \dR^d\setminus\ov{\Omega}$
being the exterior of a bounded open set~$\Omega$.
The main result of~\cite{KL1} reads as follows:
\setcounter{thm}{-1}
\begin{thm}[{\cite[Thm.~1]{KL1}}]\label{thm0}
Let $d=2$.
	For all negative~$\alpha$, we have
	\begin{flalign*}
		&&
		\max_{\stackrel[\Omega~\text{\rm convex}]{}{|\p\Omega| = c_1}}\Ev	
		= 
		\lm_1^\aa(B_{R_1}^\ext)
		\qquad \mbox{and} \qquad
		\max_{\stackrel[\Omega~\text{\rm convex}]{}{|\Omega| = c_2}}\Ev	
		= 
		\lm_1^\aa(B_{R_2}^\ext)
	\,.
	&&
	\end{flalign*}
	Here the maxima are taken over all convex, smooth, bounded
	planar open sets~$\Omega$ 
	of a given perimeter~$c_1 > 0$ or area~$c_2 > 0$, respectively,
	and~$B_{R_1}$ and~$B_{R_2}$ are disks of perimeter 
	$|\p B_{R_1}| = c_1$ and area $|B_{R_2}| = c_2$.
\end{thm}	
Hence, contrary to the bounded setting,
the exterior of the disk is the maximiser not only for the isoperimetric 
but also for the isochoric optimisation problem,
at least among the class of exteriors of convex sets.
The restriction to negative~$\alpha$ 
in Theorem~\ref{thm0} is due to the fact
that $\lambda_1^\alpha(\Omega^\ext)=0$ for all $\alpha$~positive
and any bounded~$\Omega$, so the optimisation problems
are not interesting for positive~$\alpha$.
In fact (\cf~\cite[Prop.~1]{KL1}), 
the whole interval $[0,\infty)$ belongs to the essential spectrum
of the Robin Laplacian in the exterior of any compact set, 
for any~$\alpha$, and there is no other spectrum if~$\alpha$ is positive.  
On the other hand, for every bounded non-empty~$\Omega$ there exists
a non-positive constant $\alpha_*(\Omega^\ext)$ such that  
$\lambda_1^\alpha(\Omega^\ext)$ is a negative discrete eigenvalue
if, and only if,
$\alpha < \alpha_*(\Omega^\ext)$.
\begin{prop}\label{Prop.critical}
Let $\Omega \subset \Real^d$ be an arbitrary 
non-empty smooth bounded open set.	
Then
$$
\begin{cases}
  \alpha_*(\Omega^\ext) = 0 
  \quad\mbox{if}\quad 
  d=1,2 \,,
  \\
  \alpha_*(\Omega^\ext) < 0 
  \quad\mbox{if}\quad 
  d \geq 3 \,.
\end{cases}
$$
\end{prop}

The proof of the proposition for $d=1,2$ can be found in~\cite[Prop.~2]{KL1}.
The case $d \geq 3$ is established below 
with help of a Gagliardo-Nierenberg-Sobolev inequality.
For the exteriors of balls, the critical coupling can be computed explicitly
by using the explicit form of solutions of $-\Delta u = \lambda u$
in terms of Bessel functions (see below).
\begin{prop}\label{Prop.ball}
Let $d \geq 2$. Then
$$
  \aa_*(B_R^\ext) = -\frac{d-2}{R}
  \,.
$$
\end{prop}

As a consequence of Proposition~\ref{Prop.critical},
we see that $\lambda_1^\alpha(\Omega^\ext)$ is a negative discrete eigenvalue
for \emph{all} negative~$\alpha$ only if $d=1,2$ and that is why  
Theorem~\ref{thm0} is restricted to the planar situation
(the case $d=1$ is trivial).
At the same time, in~\cite[Sec.~5.3]{KL1} we argue 
why the claim of Theorem~\ref{thm0} 
cannot hold in the same form 
for dimensions $d \ge 3$. 

Since~$\Omega$ of Theorem~\ref{thm0} is assumed to be convex,
it is necessarily connected.
According to an example in~\cite[Sec.~5.1]{KL1}, 
the connectedness of~$\Omega$ is necessary in the above theorem. 

The discussion in the two preceding paragraphs leads to the following natural questions
related to Theorem~\ref{thm0}:
\begin{myenum}
	\item[1.]
\emph{Can one replace the convexity assumption by connectedness?}
	\item[2.] 
\emph{Does a similar result hold in higher dimensions under  other constraints?}
\end{myenum}	

The objective of the present paper is to elaborate on these questions,
which were left open in our previous work~\cite{KL1}.
The present paper can be thus viewed as a natural continuation of~\cite{KL1},
but it can be also read fully independently.

\subsection{From convexity to connectedness}
In what follows, let $\Omega\subset\dR^d$ with $d \ge 2$ 
be a non-empty smooth bounded open set.
The set~$\Omega$ is not necessarily connected,
but we always assume that its exterior~$\Omega^\ext$ is connected.
The volume of and 
the area of the boundary for~$\Omega$ 
will be denoted by~$|\Omega|$ and~$|\p\Omega|$, respectively.
By $N_\Omega$ we denote the number of connected
	components of~$\Omega$.

In the two-dimensional setting, 
our main result reads as follows.
\begin{thm}\label{Thm1}
	Let $d=2$. For all negative~$\alpha$, we have
	\begin{equation*}
		\max_{\frac{|\p\Omega|}{ N_\Omega} = c} \Ev	\le \EvB \,.
	\end{equation*}
	Here the maximum is taken over all smooth, bounded
	open sets~$\Omega$ consisting
	of finitely many disjoint simply connected
	components and satisfying the relation 
	$\frac{|\p\Omega|}{N_\Omega} = c$ with given $c> 0$
	and  $B_R$ is the disk of 
	perimeter $|\p B_R| = c$.
\end{thm}

The main improvement upon Theorem~\ref{thm0}
consists in the replacement of exteriors of convex sets 
by exteriors of finite unions of simply connected (not necessarily convex) components.
In Section~\ref{ssec:ex1} we 
use large coupling asymptotics (\ie~$\aa\arr -\infty$)
to argue why the result of Theorem~\ref{Thm1}
is sharp even for a subclass of sets having
prescribed fixed number of connected components. 
The proof of Theorem~\ref{Thm1} 
relies on a usage of parallel coordinates
as employed by Payne and Weinberger in~\cite{Payne-Weinberger_1961} 
in order to get an upper bound on the principal Dirichlet eigenvalue 
on bounded domains (see~\cite{FK7,KL1} for previous applications
of the technique in the Robin problem).

In the special case of simply connected domains
the statement of Theorem~\ref{Thm1} implies the following important 
improvement upon Theorem~\ref{thm0}.
\begin{cor}\label{Cor}
Let $d=2$.
	For all negative~$\alpha$, we have
	\begin{equation*}
		\max_{|\p\Omega| = c_1}	\Ev	
		= 
		\lm_1^\aa(B_{R_1}^\ext)
		\qquad \mbox{and} \qquad
		\max_{|\Omega| = c_2}\Ev	= \lm_1^\aa(B_{R_2}^\ext)
                \,. 
	\end{equation*}
	Here the maxima are taken over all smooth, bounded,
	simply connected open sets $\Omega$ 
	of a given perimeter $c_1 > 0$ or area $c_2 > 0$, respectively,
	and $B_{R_1}$ and $B_{R_2}$ are disks of perimeter $|\p B_{R_1}| = c_1$	and area $|B_{R_2}| = c_2$.
\end{cor}

\subsection{Higher dimensions}
Let us now pass to the discussion of our results in higher dimensions. 
To this aim we need to recall some geometric concepts.
Let $\kp_1,\kp_2,\dots, \kp_{d-1}$ denote the principal curvatures of $\p\Omega$;
our convention is such that these functions are non-negative if~$\Omega$ is convex.
The mean curvature of $\p\Omega$ is defined as the function
\begin{equation}\label{eq:mean}
	M := \frac{\kp_1+ \kp_2+\dots + \kp_{d-1}}{d-1}
	\,.
\end{equation}
For the notational convenience,
we define also the number
\begin{equation}\label{total}
	\cM(\p\Omega) := 
        \frac{1}{|\p\Omega|}
	\dint[]{\p\Omega} M^{d-1}  \,.
\end{equation}
Note that for the ball $B_R\subset\dR^d$ of radius $R > 0$ 
we have $\cM(\p B_R)  = R^{-(d-1)}$.
Now we are prepared to formulate our main result in higher dimensions.
\begin{thm}\label{Thm2}
Let $d \geq 3$.
	For all negative~$\alpha$, we have
	\begin{align*}
		\max_{\stackrel[\Omega~\text{\rm convex}]{}{\cM(\p\Omega) = c}}
		\Ev	= \EvB \,.
	\end{align*}
	Here the maximum is taken over all convex, smooth,
	bounded open sets~$\Omega$
   	such that $\cM(\p\Omega) = c$
   		with given $c > 0$
   		and $B_R$ is the ball with $\cM(\p B_R) = c$.
\end{thm}

Since $\cM(\p\Omega) = 2\pi/|\p\Omega|$ for any simply connected planar domain, 
Theorem~\ref{Thm2} remains valid in dimension two, 
where it follows from the isoperimetric result of Theorem~\ref{thm0}.
If $d=3$, the integral $\int_{\p\Omega} M^{2}$
in the numerator in~\eqref{total} is sometimes referred to
as the Willmore energy, due to Willmore's demonstration that
the sphere minimises this integral, \cf~\cite[Sec.~2]{Willmore}.

If $d \geq 3$ and $\alpha \geq \alpha_*(\Omega^\ext)$,
the inequality $\Ev \leq \EvB$ of Theorem~\ref{Thm2} by itself
is just a trivial statement,
for in this case $\lambda_1^\alpha(\Omega^\ext) = 0$
is the lowest point of the essential spectrum.
However, from this inequality 
immediately follows an interesting optimisation result 
for the critical coupling.
%
\begin{cor}
Let $d \geq 3$.
        We have
	\begin{align*}
		\min_{\stackrel[\Omega~\text{\rm convex}]{}{\cM(\p\Omega) = c}}
		\aa_*(\Omega^\ext) = \aa_*(B_R^\ext) \,.
	\end{align*}
	Here the minimum is taken over all convex, smooth,
	bounded open sets~$\Omega$
	such that $\cM(\p\Omega) = c$
	with  given $c > 0$
		and $B_R$ is the ball with $\cM(\p B_R) = c$.
\end{cor}

For the proof of Theorem~\ref{Thm2} we push forward the approach of~\cite{KL1} 
based on the parallel coordinates. 
This method requires the convexity assumption, as otherwise the parallel
coordinates on $\Omega^\ext$ are not well defined.

\subsection{Organisation of the paper}
This paper is organised as follows.
In Section~\ref{Sec.pre} we provide an operator-theoretic framework for 
the Robin eigenvalue problem in the exterior of a compact set. 
Namely, we recall basic spectral properties
from~\cite{KL1} and obtain some new ones. In particular, 
we prove Proposition~\ref{Prop.critical} for $d \geq 3$
and Proposition~\ref{Prop.ball}.
The two-dimensional
case is discussed in Section~\ref{Sec.2D}, in which we prove Theorem~\ref{Thm1} and its Corollary~\ref{Cor}. Moreover, we argue why the result of Theorem~\ref{Thm1} 
is sharp for domains with fixed number of connected components. 
Finally, in Section~\ref{Sec.nD} we define higher order mean curvatures and use them to prove Theorem~\ref{Thm2} on the higher-dimensional 
case. We conclude Section~\ref{Sec.nD}
by a discussion of  a connection between
Theorem~\ref{Thm2} and large coupling asymptotics.

\section{The spectral problem in the exterior of a compact set}\label{Sec.pre}
%
Throughout this section, 
$\Omega$~is an arbitrary bounded open set in~$\dR^d$ with $d \ge 2$.
While~$\Omega$ is not assumed to be connected,
a standing assumption is that the exterior~$\Omega^\ext$ is connected.
We also assume that the boundary~$\p\Omega$ is of class $C^\infty$.
Finally, $\aa$~stands for an arbitrary negative real number.

We are interested in the eigenvalue problem
\begin{equation}\label{problem.ext}
\vast \{
\begin{aligned}
	-\Delta u &= \lm u && \mbox{in} \quad \Omega^\ext \,,
	\\
	\frac{\p u}{\p n} & = \aa \;\! u && \mbox{on} 
	\quad \p\Omega^\ext \,,
\end{aligned}
\end{equation}
where~$n$ is the \emph{outer} unit normal to~$\Omega$.
As usual, we understand~\eqref{problem.ext} as the spectral problem for the self-adjoint operator $-\Delta_\aa^{\Omega^\ext}$ in $L^2(\Omega^\ext)$
associated via the first representation theorem~\cite[Thm.~VI.~2.1]{Kato} with the closed, densely defined, symmetric, and lower semi-bounded quadratic form
\begin{equation}\label{form}
	Q_\aa^{\Omega^\ext}[u] 
	:= 
        \int_{\Omega^\ext} |\nabla u|^2  			
	+\aa \int_{\partial\Omega} |u|^2
	\,, \qquad 
	\Dom(Q_\aa^{\Omega^\ext}) := W^{1,2}(\Omega^\ext)
	\,.
\end{equation}
The boundary term is understood in the sense of traces
$W^{1,2}(\Omega^\ext) \hookrightarrow L^2(\p\Omega)$
and represents a relatively bounded perturbation of 
the Neumann form~$Q_0^{\Omega^\ext}$
with the relative bound equal to zero.
Since~$\Omega$ is smooth, 
the operator domain of $-\Delta_\aa^{\Omega^\ext}$ 
consists of functions $u \in W^{2,2}(\Omega^\ext)$,
which satisfy the Robin boundary conditions from~\eqref{problem.ext} in the sense of traces and the operator acts as the distributional Laplacian 
(see, \eg, \cite[Thm.~3.5]{Behrndt-Langer-Lotoreichik-Rohleder} 
for the $W^{2,2}$-regularity).
We call $-\Delta_\aa^{\Omega^\ext}$
the \emph{Robin Laplacian} in~$\Omega^\ext$.
By the minimax principle, 
the number $\lambda_1^\alpha(\Omega^\ext)$ defined in~\eqref{Rayleigh} 
coincides with 
the lowest point in the spectrum of $-\Delta_\aa^{\Omega^\ext}$.

Let us now recall some qualitative spectral
properties of $-\Delta_\aa^{\Omega^\ext}$.
Since the boundary~$\partial\Omega$ is bounded, 
the essential spectrum coincides with the essential
spectrum of the Laplacian in the whole space
(\cf~\cite[Prop.~1]{KL1}):
\begin{equation}\label{spec.ess}
  \sess(-\Delta_\aa^{\Omega^\ext}) = [0,\infty)
  \,.
\end{equation}
If~$\alpha$ were non-negative, then the spectrum of $-\Delta_\aa^{\Omega^\ext}$
would be exhausted by the essential spectrum
and thus $\lambda_1^\alpha(\Omega^\ext)=0$.
Since~$\alpha$ is assumed to be negative, however, 
there may exist negative discrete eigenvalues,
depending on the largeness of~$\alpha$ as regards the dimension.
For low dimensions $d=1,2$ (\cf~\cite[Prop.~2]{KL1}), 
the number $\lambda_1^\alpha(\Omega^\ext)$ is negative
whenever~$\alpha$ is negative and represents therefore 
the lowest discrete eigenvalue of~$-\Delta_\aa^{\Omega^\ext}$.  
For dimensions $d \geq 3$, the coupling parameter~$\alpha$
must be sufficiently negative to make the discrete spectrum exist.
This claim is the content of Proposition~\ref{Prop.critical} above
that we prove now.

\begin{proof}[Proof of Proposition~\ref{Prop.critical}, case $d \geq 3$]
Let~$\phi$ be any test function from $C^\infty_0(\ov{\Omega^\ext})$,
a space dense in the form domain $\Dom(Q_\aa^{\Omega^\ext})$.

By the Gagliardo-Nierenberg-Sobolev-type inequality
	for exterior domains proven in~\cite[Prop.~5.1]{Lu-Ou_2005},
there exists a positive constant~$C$ such that, 
$$
  \int_{\Omega^\ext}|\nabla\phi(x)|^2 \, \dd x 
		\ge C\left (
				\int_{\Omega^\ext}|\phi(x)|^{\frac{2d}{d-2}} \, \dd x
			\right )^{\frac{d-2}{d}}.
$$ 
 Let
	us introduce an auxiliary weight function
	$\dR^d \ni x\mapsto \omega(x) :=  \exp(-|x|)$.
By the H\"older inequality, it follows
$$
   \int_{\Omega^\ext}|\nabla\phi(x)|^2 \, \dd x 
		\ge C \
  \frac{\dint[]{\Omega^\ext}|\phi(x)|^2 \, \omg(x) \, \dd x }
  {\left (\dint[]{\Omega^\ext} \, \omg(x)^{\frac{d}{2}}	\, \dd x\right )^{\frac{2}{d}}}
  = \frac{C d^2}{4\big(\Gamma(d) s_d\big)^{\frac{2}{d}}} 
  \dint[]{\Omega^\ext}|\phi(x)|^2 \, \omg(x) \, \dd x
  \,,
$$
where~$\Gamma$ denotes the Euler Gamma-function
and 
$
  s_d := |\partial B_1| = \frac{2\pi^{d/2}}{\Gamma(d/2)}
$ 
stands for the area of the unit sphere.
Note  that the obtained bound is a Hardy-type inequality for the Neumann Laplacian in $\Omega^\ext$. 

Now,  let $B_R\subset\dR^d$ be an open ball of a sufficiently large radius $R > 0$, so that $\ov{\Omega}\subset B_R$. Furthermore, we define the domain 
$\Omega_0^\ext := \Omega^\ext\cap B_R$.
Notice that the function $\omg$ is uniformly positive
in $\Omega_0^\ext$. Hence, the estimate above yields 
that there exists a positive constant~$c$ such that,
	\[
		\int_{\Omega^\ext}|\nabla\phi(x)|^2 \, \dd x 
		\ge 
		c \left (\int_{\Omega_0^\ext}|\nabla\phi(x)|^2\, \dd x 
		+ \int_{\Omega_0^\ext} |\phi(x)|^2 \, \dd x \right ) \,.
	\]
	Eventually, the trace theorem~\cite[Thm.~3.38]{McLean} 
	implies that there exists another positive constant~$c'$ such that,
	\[
		\int_{\Omega^\ext}|\nabla\phi(x)|^2 \, \dd x 
		\ge 
		c'
		\int_{\p\Omega}|\phi(x)|^2 \, \dd \s(x)
                \,.
	\]
Since~$\phi$ is an arbitrary function from the form core of~$-\Delta_\aa^{\Omega^\ext}$,
this inequality and~\eqref{spec.ess}
show that $\lambda_1^\alpha(\Omega^\ext)=0$
for all $\alpha \geq -c'$.
Consequently, the critical constant $\alpha_*(\Omega^\ext)$ 
for which a discrete eigenvalue emerges from the essential spectrum	
must be negative.

To show that $\alpha_*(\Omega^\ext)$ is actually finite 
and that the discrete spectrum exists for all $\alpha < \alpha_*(\Omega^\ext)$,
it is enough to notice the validity of the form ordering 
$
    Q_{\aa_1}^{\Omega^\ext}[\phi] \le
	Q_{\aa_2}^{\Omega^\ext}[\phi]
$ 
for $\aa_1 \le \aa_2$
and the continuity of $\aa\mapsto Q_{\aa}^{\Omega^\ext}[ \phi]$.
Consequently, the minimax principle then implies that 
$\aa\mapsto\Ev$ is non-decreasing and continuous.
Hence, the desired claim follows.
\end{proof}	

For $d=2$ it was shown in~\cite[Prop.~5]{KL1} that,
for all negative~$\alpha$,
\begin{equation}\label{monotonicity}
\mbox{$R \mapsto \EvB$ is strictly decreasing}. 
\end{equation}
Obviously, because of Proposition~\ref{Prop.critical},
the same monotonicity result cannot hold for \emph{all} negative~$\alpha$
if $d \geq 3$. 
In the higher dimensions, a monotonicity of
the critical coupling $\alpha_*(B_R^\mathrm{ext})$
for the exteriors of the balls~$B_R$ 
follows from Proposition~\ref{Prop.ball} above 
that we prove now.

\begin{proof}[Proof of Proposition~\ref{Prop.ball}]
	 In view of the radial symmetry
	of the problem,	the eigenfunction $u_1\in W^{1,2}(B_R^\ext)$
	of the Robin Laplacian in the exterior of the ball~$B_R$
corresponding to its lowest eigenvalue
	$\EvB < 0$, if it exists, must necessarily be radially symmetric as well. Using this simple observation we see that $\EvB = -k^2 < 0$ \iff the following ordinary differential spectral problem
	\begin{equation}\label{ODE}
		\Vast \{ 
		\begin{aligned}
			-r^{-(d-1)}\big[r^{d-1} \psi'(r)]' & = - k^2\psi(r)
				&& \mbox{for}\quad  r\in (R,\infty)\,,\\[0.3ex]
			\psi'(R) &= \aa\;\!\psi(R)\,,\\[0.3ex]
			\lim_{r\arr\infty}\psi(r) &= 0\,,
		\end{aligned}
	\end{equation}
	possesses a solution $(\psi, k)$ with
	$\psi\neq 0$ and $k > 0$; \cf~\cite[Sec.~3]{FK7}.
	Observe that the general solution of the differential equation in~\eqref{ODE}
with $k > 0$ is given by
	\[
		\psi(r) = 
		r^{-\nu}\big[C_1K_\nu(kr) +  C_2 I_\nu(kr)\big],
		\qquad 
		C_1,C_2 \in\dC,\qquad \nu := \frac{d-2}{2}\,,
	\]
	where $K_\nu(\cdot)$ and $I_\nu(\cdot)$ are modified Bessel functions of the second kind and order~$\nu$.
	Taking into account the required decay at infinity from~\eqref{ODE}
and using the asymptotic behaviour of $K_\nu(x)$ and $I_\nu(x)$ 
for large~$x$, see~\cite[Sec.~9.7.2]{Abramowitz-Stegun},
	we conclude that $C_2 = 0$.
Thus, the expression for $\psi$ simplifies to
$
		\psi(r) = C_1r^{-\nu}K_\nu(kr)
$,
where the constant~$C_1$ should be non-zero 
to get a non-trivial solution.
	Differentiating~$\psi$ with respect to $r$, 
we find
	\[
		\psi'(r) 
		= 
		-C_1\nu r^{-\nu - 1} K_\nu(kr) 
							+ C_1k r^{-\nu} K_\nu'(kr)\,.
	\]
	Thus, the boundary condition from~\eqref{ODE}
	yields the requirement
	\[
		-\nu R^{-\nu - 1} K_\nu(kR) 
				+ k R^{-\nu} K_\nu'(kR)
					-\aa R^{-\nu}K_\nu(kR) = 0\,.
	\]

	With the aid of the identity
	$K_\nu'(x) =- K_{\nu+1}(x) + \frac{\nu}{x} K_\nu(x)$
	(see \cite[Sec.~9.6.26]{Abramowitz-Stegun}),
	this scalar equation simplifies to
	\begin{equation}\label{eq:algebraic}
		\aa = -\frac{\nu}{R} + k \frac{K_\nu'(kR)}{K_\nu(kR)}
		=  
		-\frac{\nu}{R} + \frac{\nu}{R} - k \frac{K_{\nu+1}(kR)}{K_\nu(kR)}
		= 
		- kR \frac{K_{\nu+1}(kR)}{K_\nu(kR)} \frac{1}{R}.
	\end{equation}
	Introduce now $f(x) :=  x\frac{K_{\nu+1}(x)}{K_\nu(x)}$.
	The function $f$ is clearly continuous on $[0,\infty)$
	and the equation~\eqref{eq:algebraic} rewrites as
	$-\aa R = f(kR)$.
	Using the identities $2K_\nu' =- K_{\nu+1} - K_{\nu-1}$
	and $K_{\nu+1}(x) - K_{\nu-1}(x) = \frac{2\nu}{x}K_\nu(x)$
	(see \cite[Sec.~9.6.26]{Abramowitz-Stegun}), 
we find after lengthy but elementary
	computations that
(here the derivative is with respect to~$x$ 
and for brevity we omit the arguments of the functions)
	\[
	\begin{split}
		 f'
		& = 
		\frac{K_{\nu+1}}{K_\nu} +
		\frac{x}{K_\nu^2}\left(K_\nu K_{\nu+1}' - K_\nu' K_{\nu+1}\right)\\
		& =
		\frac{K_{\nu+1}}{K_\nu}+  \frac{x}{2K_\nu^2}\left(
		K_{\nu-1} K_{\nu+1} + K_{\nu+1}^2 - K_\nu^2 - K_\nu K_{\nu+2}\right)\\
		& =
		\frac{K_{\nu+1}}{K_\nu}
		+
		\frac{x}{2K_\nu^2}\left[
			K_{\nu+1}^2 - K_\nu^2 
			\right]\\
		& \qquad\quad +
		\frac{x}{2K_\nu^2}\left[
		\left(K_{\nu+1} - \frac{2\nu K_\nu}{x}\right) 
		K_{\nu+1} - K_\nu \left(K_\nu + 
		\frac{2(\nu+1)K_{\nu+1}}{x}\right)
		\right] \\
		& =
		-2\nu \frac{K_{\nu+1}}{K_\nu}
		+
		\frac{x}{K_\nu^2}\left[	K_{\nu+1}^2 - K_\nu^2\right] 
		 =
		\frac{x}{K_\nu^2}
		\left[
			-\frac{2\nu}{x} K_{\nu}K_{\nu+1}
			+
			K_{\nu+1}^2 - K_\nu^2
		\right]\\
		& =
		\frac{x}{K_\nu^2}
		\left[
			(K_{\nu-1} - K_{\nu+1}) K_{\nu+1}
			+
			K_{\nu+1}^2 - K_\nu^2
		\right]
		= 
		x
		\left[\frac{K_{\nu-1} K_{\nu+1}}{K_\nu^2} - 1\right] > 0\,,
	\end{split}	
	\]
	where the last step follows from the inequality in~\cite[Thm.~8]{Segura_2011}.
	We have thus shown that $f' > 0$ and hence the function~$f$  is strictly increasing.	In view of the asymptotic expansion 
	$K_\nu(x) \sim \frac{\Gamma(\nu)}{2} \left(\frac{2}{x}\right)^\nu$ as $x \arr 0^+$ (see \cite[Sec.~9.6.9]{Abramowitz-Stegun}), we obtain
	\begin{equation}\label{eq:lim1}
		\lim_{x\arr 0^+} f(x) 
		= 
		\frac{2^{\nu+1}\Gamma(\nu+1)}{2} 
		\left(\frac{2^{\nu}\Gamma(\nu)}{2}\right)^{-1}
		= 
		2\frac{\Gamma(\nu+1)}{\Gamma(\nu)} 
		=
		2\nu\,.
	\end{equation}
	Using $K_\nu(x) \sim \left(\frac{\pi}{2x}\right)^{1/2} e^{-x}$
	as $x\arr\infty$ (see \cite[Sec.~9.7.2]{Abramowitz-Stegun}),
 we also find
	\begin{equation}\label{eq:lim2}
		\lim_{x\arr \infty} f(x) = \infty\,.
	\end{equation}
	Finally, combining monotonicity of $f$ and the limits~\eqref{eq:lim1},
	~\eqref{eq:lim2}, we conclude that the algebraic equation
	$-\aa R = f(k R)$ has at least one (and exactly one) solution $k > 0$ \iff $\aa < -\frac{2\nu}{R} = -\frac{d-2}{R}$. 
\end{proof}	

%


\section{Planar domains}\label{Sec.2D}

\subsection{Proof of Theorem~\ref{Thm1}}\label{Sec.proof}
%
Let $N := N_\Omega\in\dN$ be the number
	of connected components of the domain 
	$\Omega = \cup_{n=1}^N \Omega_n\subset\dR^2$,
where $\Omega_n\subset\dR^2$ are bounded, smooth, simply connected domains 
and $\ov{\Omega_n}\cap\ov{\Omega_m} = \varnothing$
if $n\ne m$. Let $\kp_n\colon\p\Omega_n\arr\dR$ 
be the curvature of the curve $\p\Omega_n$,
with the sign convention that $\kappa_n$ is non-positive if~$\Omega_n$ is convex. 
By \cite[Cor.~2.2.2]{Kli}, we have $\int_{\p\Omega_n} \kp_n = -2\pi$ for all $n = 1,\dots, N$. 
According to the constraint in the formulation of the
	theorem, we also have $\frac{|\p\Omega|}{N} = |\p B_R| =  c > 0$.

Let $\rho\colon \Omega^\ext\mapsto (0,+\infty)$ be the distance
function from the boundary $\p\Omega$ of $\Omega$. 
Furthermore, we define one more auxiliary function by
\[
	A \colon [0,\infty) \arr [0,\infty) ,\qquad 
	A(r) := \big|\{x\in\Omega^\ext \colon \rho(x) < r\}\big|.
\]
Note that the value $A(r)$ is simply the area of the sub-domain of $\Omega^\ext$ 
which consists of the points located 
at a distance less than~$r$ from the boundary~$\p\Omega$. According to~\cite[Prop. A.1] {Savo_2001}, the function $A(r)$ is
	locally Lipschitz continuous and thus
	differentiable almost everywhere.

For a Lipschitz continuous, 
compactly supported $\phi \colon [0,\infty)\arr\dC$,
we introduce the compositions 
$u := \phi \circ A \circ \rho\colon \ov{\Omega^\ext}\arr\dC$ 
and $\psi := \phi\circ A\colon [0,\infty)\arr\dC$.
Hence, $u$ and $\psi$ are Lipschitz continuous
and compactly supported in $\ov{\Omega^\ext}$ and
in $[0,\infty)$, respectively.  In particular,
we have	$u\in W^{1,2}(\Omega)$ and
\[
\begin{split}
	\|\nabla u\|^2_{L^2(\Omega;\dC^2)} 
	& 
	= 
	\int_0^\infty |\phi'(A(r))|^2 (A'(r))^3 \, \dd r 
	= 
	\int_0^\infty |\psi'(t)|^2 A'(t) \, \dd t\,,\\[0.6ex]
	\|u\|^2_{L^2(\Omega)} & 
	= 
	\int_0^\infty |\phi(A(r))|^2A'(r) \, \dd r
	= 
	\int_0^\infty |\psi(t)|^2 A'(t) \, \dd t\,,\\[0.6ex]
	\| u \|^2_{L^2(\p\Omega)} & 
	= 
	 |\p\Omega| \, |\phi(0)|^2 
	= 
	|\p\Omega| \, |\psi(0)|^2 \,.
\end{split}
\]
The last formula of the above three is almost obvious. The first and the second formulae are consequences of the co-area formula
and their complete derivation can be found in~\cite[App.~1]{Savo_2001}
(see also~\cite[Sec.~4]{FK7}).

Furthermore, we observe using geometric isoperimetric
inequality and~\cite[Prop.~A.1\,(iv)]{Savo_2001}
(see also~\cite{Nagy_1959}) that
\[
\begin{aligned}
	(4\pi|\Omega|)^{1/2}
	\le A'(r) = 
	\left|\{x\in\Omega^\ext \colon \rho(x) = r\}\right|
	&\le |\p\Omega| - r\sum_{n=1}^N \int_{\p\Omega_n} \kp_n(s) \, \dd s 
        \\
	&= 
	|\p\Omega| + 2\pi N r
   \,.
\end{aligned}
\]
Note that any
$\psi \in C^\infty_0([0,\infty))$
can be represented as a composition $\phi \circ A$ 
with a Lipschitz continuous, compactly
supported $\phi = \psi \circ A^{-1}
 \colon [0,\infty)\arr\dC$,  where $A^{-1}$ stands 
 for the function inverse to $A$.
Thus, by the minimax principle~\cite[Sec.~XIII.1]{RS4} applied
for the quadratic form $Q^{\Omega^\ext}_\aa$
on the subspace of Lipschitz continuous functions, compactly supported in the closure of $\Omega^\ext$ and
depending on the distance from its boundary only, we get
\[
\begin{split}
	\Ev	
	&
	\le 
	\inf_{\begin{smallmatrix}\psi\in C^\infty_0([0,\infty))\\ \psi \ne 0 \end{smallmatrix}}
		\frac{\dint[\infty]{0} |\psi'(t)|^2 ( |\p \Omega| + 2\pi N t) \, \dd t+ \aa \, |\p \Omega| |\psi(0)|^2}{\dint[\infty]{0}|\psi(t)|^2 (|\p \Omega|+2\pi N t) \, \dd t}\\
	& =
	\inf_{\begin{smallmatrix}\psi\in C^\infty_0([0,\infty))\\ \psi \ne 0 \end{smallmatrix}}
	\frac{\dint[\infty]{0} |\psi'(t)|^2\left(\frac{|\p \Omega|}{N}+2\pi t\right) \dd t+ \aa \, \frac{|\p \Omega|}{N} \, |\psi(0)|^2}{\dint[\infty]{0} |\psi(t)|^2\left (\frac{|\p \Omega|}{N}+2\pi t \right)  \dd t}
  \\
	&=
	\inf_{\begin{smallmatrix}\psi\in C^\infty_0([0,\infty))\\ \psi \ne 0 \end{smallmatrix}}
	\frac{\dint[\infty]{0} |\psi'(t)|^2\left(|\p B_R|+2\pi t\right) \dd t+ \aa \, |\p B_R| \, |\psi(0)|^2}{\dint[\infty]{0} |\psi(t)|^2\left (
		|\p B_R| +2\pi t \right)  \dd t} 
        \\ 
        &= 
	 \EvB\,,
\end{split}	
\]
where in the last step we implicitly employed that $\EvB < 0$ 
and that the eigenfunction of the Robin Laplacian $\OpB$ 
corresponding to $\EvB$ is radially symmetric.
\qed

\subsection{Large coupling for a union of identical disjoint disks}
\label{ssec:ex1}
The result of Theorem~\ref{Thm1} can be viewed
as an inequality 
\begin{equation}\label{eq:ineqThm1}
	\Ev \le \EvB\,,\qquad \forall\, \aa < 0\,,
\end{equation}
for $\Omega\subset\dR^2$ being
a smooth, bounded open set consisting of 
$N\in\dN$ disjoint simply connected components and satisfying
the relation $\frac{|\p\Omega|}{N} = |\p B_R| = c > 0$. In this
section we argue why the inequality~\eqref{eq:ineqThm1}
can not be improved for a subclass of domains 
with prescribed number of connected components.

To this aim fix a positive number~$r$ and a discrete set of points
$X := \{x_n\}_{n=1}^N \subset \Real^2$
such that $|x_n - x_m| > r$ for any $n\ne m$.
Let us consider the planar set
\[
\Omega := \bigcup_{n=1}^N B_r(x_n)  
\,,
\]	 
where $B_r(x_n)$ denotes the open disk of radius~$r$ centered at~$x_n$.  
Clearly, we have $N_\Omega = N$, $c = \frac{|\p\Omega|}{N} =   2\pi r$,
and the curvature of $\p\Omega$ equals $-1/r$ pointwise.
Using the large coupling asymptotics given 
in~\cite[Cor.~1.4]{Pankrashkin-Popoff_2016},
we have 
\[
\begin{aligned}
	\lm_1^\aa(\Omega^\ext) & 
		= -\aa^2 -\frac{\aa}{r} +  o(\aa),
	\qquad
	&\aa\arr -\infty\,,\\
	\lm_1^\aa(B^\ext_{R})& 
		= -\aa^2 -\frac{\aa}{R} + o(\aa),
	\qquad
	&\aa\arr -\infty\,.
\end{aligned}	
\]
We conclude that for $R > r$ the ``reversed'' inequality
$\Ev >\EvB$ holds for all $\aa < 0 $ with $|\aa|$ large enough. 
Hence, the inequality~\eqref{eq:ineqThm1} is in general not valid
for $|\p B_R| > c$ for a subclass of domains with
	exactly~$N$ connected components.

\subsection{Proof of Corollary~\ref{Cor}}
The optimisation result under the fixed perimeter constraint immediately
follows from Theorem~\ref{Thm1}. In order
to show the optimisation result under fixed area constraint,
we first observe that, by the isoperimetric result,
$\Ev \le \lm_1^\aa(B_{R_2}^\ext)$
if $|\p\Omega| = |\p B_{R_2}|$. Next, note that by the geometric isoperimetric inequality for the ball $B_{R_1}$ satisfying $|\Omega| = |B_{R_1}|$ we have $R_1 \le R_2$ 
and thus by the strict monotonicity~\eqref{monotonicity}
we get, for any negative~$\aa$,
\begin{flalign*}
	&&\Ev \le \lm_1^\aa(B_{R_2}^\ext) \le \lm_1^\aa(B_{R_1}^\ext)\,.&&\qed
\end{flalign*}

\section{Domains in higher space-dimensions}\label{Sec.nD}

\subsection{Higher order mean curvatures}
Let $\Omega\subset\dR^d$, with $d \ge 3$, be a bounded
smooth domain. Recall that $\kp_1,\kp_2,\dots, \kp_{d-1}$ 
denote the principal curvatures of $\p\Omega$. 
They are defined locally as eigenvalues of
the \emph{Weingarten tensor} $\mathcal{W} := \der n$,
where~$n$ is the outer unit normal to~$\Omega$ in our convention.
Consequently, the principal curvatures are non-negative if~$\Omega$ is convex.

Given $j \in \{0,1,\dots,d-1\}$,
let  $M_j$ be the \emph{$j^\mathrm{th}$-order mean curvature}, 
normalised so that the \emph{symmetric function of the principal curvatures}
$F_{\p\Omega}\colon \p\Omega\times [0,\infty) \arr \dR$
can be expanded as follows:
\begin{equation}\label{eq:sfpc}
F_{\p\Omega}(s,t) 
:=
\prod_{j=1}^{d-1} \big(1 + t \, \kp_j(s) \big) 
=
\sum_{j=0}^{d-1}\big( 
\begin{smallmatrix} d -1 \\ j \end{smallmatrix}\big) 
\, M_j(s) \, t^j .
\end{equation}
Thus, $M_0 = 1$, 
$M_1 = M$ is the usual mean curvature introduced already in~\eqref{eq:mean}, 
and $M_{d-1} = \prod_{j=1}^{d-1}\kp_j$ is the \emph{Gauss-Kronecker curvature}. 
While the principal curvatures $\kappa_1, \dots, \kappa_{d-1}$
are defined only locally,
the invariants $M_1,\dots,M_{d-1}$ are globally defined functions.

The averaged $j^\mathrm{th}$-order integral of the mean curvature is defined by
\[
	\cM_j(\p\Omega) 
	:= 
	\frac{1}{|\p\Omega|} \int_{\p\Omega} M_j(s) \, \dd \s(s)\,,
\]
for $j\in \{1,\dots,d-1\}$. 
By~\cite[\S 4.2]{Sch} we have
$\cM_{d-1}(\p\Omega) = \frac{s_d}{|\p\Omega|}$ 
for any convex set~$\Omega$, 
where~$s_d$ as usual denotes the area of the unit sphere in $\dR^d$. Note also that $\cM_j(\p B_R) = R^{-j}$ for all $j=1,\dots,d-2$.

\subsection{Proof of Theorem~\ref{Thm2}}
Let $\Omega\subset\dR^d$, $d \ge 3$, be a bounded, convex smooth domain and let $B_R\subset\dR^d$ be a ball such that
$\cM(\p\Omega) = \cM(\p B_R) = R^{-(d-1)}$.
In the case that $\EvB = 0$
the claim obviously holds
and we assume without loss of generality that
$\aa < 0$ and $R > 0$ are such that $\EvB < 0$.
We split the proof into four steps.

\medskip
\underline{\emph{Step 1.}}
For any $j \in \{1,\dots, d-2\}$, 
Maclaurin's inequality~\cite[Ineq.~52]{HLP}
yields that $M_j \le M^j$ holds pointwise.
Hence, for $j \in \{1,\dots, d-2\}$ 
we get using Jensen's inequality applied to
the concave function $[0,\infty) \ni x\mapsto x^{\frac{j}{d-1}}$ 
\[
\begin{split}
	\cM_j(\p\Omega) & = 
	\frac{1}{|\p\Omega|} \int_{\p\Omega} M_j(s) \, \dd \s(s)\\
	& \le 
	\frac{1}{|\p\Omega|} 
	\int_{\p\Omega} (M(s))^j \, \dd \s(s)\\
	& \le 
	\left (\frac{1}{|\p\Omega|} 
	\int_{\p\Omega} (M(s))^{d-1} \, \dd \s(s)
	\right )^{\frac{j}{d-1}}\\
	& 
	= (\cM(\p\Omega))^{\frac{j}{d-1}}
	= (\cM(\p B_R))^{\frac{j}{d-1}}
	= \cM_j(\p B_R)\,.
\end{split}
\]
In particular, we have shown that
\begin{equation*} 
	\cM_1(\p\Omega) \le \cM_1(\p B_R) = \frac{1}{R}\,.
\end{equation*}

\medskip

\underline{\emph{Step 2.}}
In this step we show that $|\p\Omega| \ge |\p B_R|$. 
Using~\cite[Eq.~17 in \S~19.3, Rem.~19.3.4]{BuZa} and the Alexandrov-Fenchel inequality~\cite[\S~20.2, Eq.~20]{BuZa} we obtain
\[
	\cM_1(\p\Omega)
	 \ge 
	 \left (
	 \frac{s_d}{|\p\Omega|} \right )^{\frac{1}{d-1}} \,.
\]
Hence, the inequality $\cM_1(\p\Omega) \le R^{-1}$ 
(shown in Step 1) yields
\[
	|\p\Omega| 	\ge  s_d R^{d-1}
	=
	|\p B_R|\,.
\]

\medskip
\underline{\emph{Step 3.}}
%
Integrating the symmetric functions $F_{\p\Omega}$
and $F_{\p B_R}$ of the principal
curvatures defined as in~\eqref{eq:sfpc} over $\p\Omega$
and $\p B_R$, respectively,
we obtain \emph{Steiner-type polynomials}
\[
\begin{split}
	P_{\p\Omega}(t) 
	& := 
	\int_{\p \Omega} \frac{F_{\p\Omega}(s,t)}{|\p\Omega|} \, \dd \s(s) 
	 = 
	1 + \sum_{j=1}^{d-2} \big(\begin{smallmatrix} d -1 \\ j \end{smallmatrix}\big) \cM_j(\p\Omega) \, t^j+ \frac{s_d \, t^{d-1}}{|\p\Omega|}\,, \\
	P_{\p B_R}(t) 
	& := 
	\int_{\p B_R} \frac{F_{\p B_R}(s,t)}{|\p B_R|} \, \dd \s(s) 
	= 
	1 + \sum_{j=1}^{d-2} \big(\begin{smallmatrix} d -1 \\ j \end{smallmatrix}\big) \cM_j(\p B_R) \, t^j+ \frac{s_d \, t^{d-1}}{|\p B_R|}\,.
\end{split}	
\]
Further, using the inequalities 
$\cM_j(\p\Omega)\le \cM_j(\p B_R)$, $j=1,\dots,d-2$, and $|\p B_R| \le |\p\Omega|$
(shown in Steps 1 and 2) we obtain
\begin{equation}\label{eq:Pineq}
	P_{\p\Omega}(t) \le P_{\p B_R}(t)\,,\qquad \forall\, t >0\,.
\end{equation}

\medskip
\underline{\emph{Step 4.}}	
	Next, we parameterise $\Omega^\ext$
	by means of the \emph{parallel coordinates}
	\begin{equation}\label{tube} 
		\cL \colon \p\Omega \times (0,\infty) \to \Omega^\ext :
		\left\{(s,t) \mapsto s + n(s) \, t \right\}	\,,
	\end{equation}
	where~$n$ is the outer unit normal to~$\Omega$ as above. 
	Notice that~$\cL$ is indeed a global diffeomorphism
	because of the convexity and smoothness assumptions;
	\cf~\cite[Sec. 4]{KL1}.
	The metric induced by~\eqref{tube} 
	acquires a block form
	\begin{equation}\label{metric} 
		\dd\cL^2 = g \circ (I + t\, \cW)^2 + \dd t^2
		\,,
	\end{equation}
	where~$g$ is the Riemannian metric of~$\partial\Omega$
	and~$\cW=\der n$ is the Weingarten tensor introduced above,
	\cf~\cite[Sec.~2]{KRT} for details.
	Hence, $\Omega^\ext$ can be identified 
	with the product manifold $\p\Omega \times (0,\infty)$	equipped with the metric~\eqref{metric}.
	Consequently, the Hilbert space $L^2(\Omega^\ext)$ 
	can be identified with 
	\[
		\cH :=
		L^2\big(\p\Omega \times (0,\infty),
		F_{\p\Omega}(s,t) 
		\, \dd \s(s) \, \dd t\big)\,,
	\]
	via the unitary transform
	\begin{equation}\label{eq:U}
		\sfU \colon L^2(\Omega^\ext) \to \cH \colon
		\{u \mapsto u \circ \cL \}
	\,.
	\end{equation}
	It is natural to introduce the unitarily equivalent operator
	$\sfH_\aa := \sfU (\Op) \sfU^{-1}$,
	associated with the transformed
	quadratic form
	$\sfh_\aa[v] :=  Q_\aa^{\Omega^\ext}[\sfU^{-1}v]
	$, $\Dom(\sfh_\aa) = \sfU(W^{1,2}(\Omega^\ext))$.	 	
	Applying the minimax principle~\cite[Sec.~XIII.1]{RS4} 
	for the quadratic form $\sfh_\aa$
	on functions in
	$C^\infty_0(\p\Omega\times [0,\infty))$
	dependent on $t$-variable only we get
	\[
\begin{split}
	\Ev 
	& \le 
	\inf_{\begin{smallmatrix}\psi \in C^\infty_0([0,\infty))\\
		\psi \ne 0
		\end{smallmatrix}}
	\frac{\dint[]{\p\Omega}\dint[\infty]{0} |\psi'(t)|^2F_{\p\Omega}(s,t) \,\dd t \,\dd \s(s) + \aa \, |\p\Omega| |\psi(0)|^2}
	{\dint[]{\p\Omega}\dint[\infty]{0} |\psi(t)|^2
		F_{\p\Omega}(s,t) \, \dd t \, \dd \s(s)}\\[0.6ex]
	& 
	=  
	\inf_{\begin{smallmatrix}\psi \in C^\infty_0([0,\infty))\\
		\psi \ne 0
		\end{smallmatrix}}
	\frac{\dint[\infty]{0}  |\psi'(t)|^2 P_{\p\Omega}(t) \, \dd t + \aa \, |\psi(0)|^2}
	{\dint[\infty]{0}  |\psi(t)|^2P_{\p\Omega}(t) \, \dd t}\\[0.6ex]
	& \le
	\inf_{\begin{smallmatrix}\psi \in C^\infty_0([0,\infty))\\
		\psi \ne 0
		\end{smallmatrix}}
	\frac{\dint[\infty]{0}  |\psi'(t)|^2 P_{\p B_R}(t) \, \dd t + \aa \, |\psi(0)|^2}
	{\dint[\infty]{0}  |\psi(t)|^2P_{\p B_R}(t) \, \dd t} \\
	&= 
	\EvB\,,
\end{split}
\]
where we have applied the inequality~\eqref{eq:Pineq} and used in between  that $\EvB < 0$
and that the eigenfunction of $\OpB$ corresponding to $\EvB$ is radially symmetric.	

\subsection{A connection with large coupling asymptotics}
The result of Theorem~\ref{Thm2} can be seen
as the inequality
\begin{equation}\label{eq:ineqThm2}
	\Ev \le \EvB\,,\qquad \forall\, \aa < 0\,,
\end{equation}
for $\Omega\subset\dR^d$, $d\ge 3$, being
a smooth, bounded convex open set 
satisfying
the relation $\cM(\p\Omega) = \cM(\p B_R)$
with $\cM(\cdot)$ defined as in~\eqref{total}. 
Assuming that $\Omega$ and $B_R$ are not
congruent, the constraint $\cM(\p\Omega) = \cM(\p B_R)$
and the main result of~\cite{Alexandrov-62} imply that
\begin{equation*}
	\delta := \min_{s\in\p\Omega} M(s) - R^{-1} < 0\,.
\end{equation*}
Hence, by~\cite[Cor. 1.4]{Pankrashkin-Popoff_2016}
we have
\begin{equation*}
	\EvB - \Ev = \aa\delta\,(d-1) + o(\aa)\,,
	\qquad\aa \arr -\infty\,.
\end{equation*}
Informally speaking, the ``gap'' in the isoperimetric
inequality~\eqref{eq:ineqThm2} grows 
asymptotically linearly in $|\aa|$
as $\aa\arr  -\infty$.

%
%
%

\subsection*{Acknowledgments}
The research of D.K.\ was partially supported by FCT (Portugal)
through project PTDC/MAT-CAL/4334/2014.
V.L.\ acknowledges the support by the grant No.~17-01706S of
the Czech Science Foundation (GA\v{C}R).

%
\bibliography{bib16}
\bibliographystyle{amsplain}

\end{document}